\documentclass[11pt]{article}
\usepackage{amsmath,amsthm,amssymb,enumerate,mathscinet}
\usepackage{fullpage}
\usepackage{pgf,tikz}
\usepackage{mathrsfs}
\usetikzlibrary{arrows}

\newtheorem{theorem}{Theorem}[section]

\newtheorem{lemma}[theorem]{Lemma}

\newtheorem{claim}[theorem]{Claim}

\newtheorem{thm}[theorem]{Theorem}

\newtheorem{obs}[theorem]{Observation}
\newtheorem{constru}[theorem]{Construction}

\def\eps{\varepsilon}

\def\COMMENT#1{}
\let\COMMENT=\footnote

\title{Large subgraphs in rainbow-triangle free colorings} 

\author{Adam Zsolt Wagner\footnote{University of Illinois at Urbana-Champaign, Urbana, Illinois 61801, USA, {\tt
zawagne2@illinois.edu}. }}
\date{}

\begin{document}
\maketitle
\begin{abstract}
Fox--Grinshpun--Pach showed that every $3$-coloring of the complete graph on $n$ vertices without a rainbow triangle contains a clique of size $\Omega\left(n^{1/3}\log^2 n\right)$ which uses at most two colors, and this bound is tight up to the constant factor. We show that if instead of looking for large cliques one only tries to find subgraphs of large chromatic number, one can do much better. We show that every such coloring contains a $2$-colored subgraph with chromatic number at least $n^{2/3}$, and this is best possible. We further show that for fixed positive integers $s,r$ with $s\leq r$, every $r$-coloring of the edges of the complete graph on $n$ vertices without a rainbow triangle contains a subgraph that uses at most $s$ colors and has chromatic number at least $n^{s/r}$, and this is best possible. Fox--Grinshpun--Pach previously showed a clique version of this result.

As a direct corollary of our result we obtain a generalisation of the celebrated theorem of Erd\H{o}s-Szekeres, which states that any sequence of $n$ numbers contains a monotone subsequence of length at least $\sqrt{n}$. We prove that if an $r$-coloring of the edges of an $n$-vertex tournament does not contain a rainbow triangle then there is an $s$-colored directed path on $n^{s/r}$ vertices, which is best possible. This gives a partial answer to a question of Loh.
\end{abstract}

\section{Introduction}

A Gallai-coloring of a complete graph is an edge coloring such that no triangle is colored with three distinct colors. Such colorings arise naturally in several areas including in information theory \cite{infotheo}, in the study of partially ordered sets, as in Gallai's original paper \cite{gallaicolpaper}, and in the study of perfect graphs \cite{perfect}. Several Ramsey-type results in Gallai-colored graphs have also emerged in the literature (see e.g. \cite{gall1}, \cite{gall2}, \cite{gall3}, \cite{gall4}), but they mostly focus on finding large monochromatic structures in such colorings.
Our main result is the observation that certain proof techniques used by Fox--Grinshpun--Pach \cite{fox} for solving the multicolor Erd\H{o}s--Hajnal conjecture for rainbow triangles also give a partial answer for Loh's question \cite{loh}, that asks the following: what is the value of $f(n,r,s)$, the maximum number such that every $r$-coloring of the edges of the transitive tournament on $n$ vertices contains a directed path with at least $f(n,r,s)$ vertices whose edges have at most $s$ colors?

\subsection{Erd\H{o}s--Hajnal for rainbow triangles}

Erd\H{o}s showed \cite{erdos} that a random graph on $n$ vertices almost surely contains no clique or independent set of order $2\log n$. On the other hand, the Erd\H{o}s--Hajnal conjecture \cite{ehajnal} states that for each fixed graph $H$ there is an $\eps=\eps(H)>0$ such that every graph $G$ on $n$ vertices which does not contain a fixed induced subgraph $H$ has a clique or independent set of order $n^\eps$, much larger than in the case of general graphs.  The Erd\H{o}s--Hajnal conjecture is still open, but there are now several partial results on it - we refer the reader to the introduction of \cite{fox} and the recent survey \cite{chudnovsky}. In the present paper we will be only interested in a special case of the multicolor generalisation of the Erd\H{o}s--Hajnal conjecture. Hajnal \cite{hajnal} conjectured that there is an $\eps>0$ such that every $3$-coloring of the edges of the complete graph on $n$ vertices without a rainbow triangle (that is, a triangle with all its edges different colors) contains a set of order $n^{\eps}$ which uses at most two colors. Fox--Grinshpun--Pach proved Hajnal's conjecture and further determined a tight bound on the order of the largest guaranteed $2$-colored set in any such coloring. A \emph{Gallai $r$-coloring} is a coloring of the edges of a complete graph using $r$ colors without rainbow triangles.

\begin{thm}[\label{fgp1}Fox--Grinshpun--Pach, \cite{fox}]
Every Gallai-$3$-coloring on $n$ vertices contains a set of order $\Omega\left(n^{1/3}\log^2n\right)$ which uses at most two colors, and this bound is tight up to a constant factor.
\end{thm}

Instead of looking for large complete subgraphs, it is also natural to try to find $2$-colored subgraphs with large chromatic number. It is easy to show that every $3$-edge-colored complete graph on $n$ vertices contains a $2$-colored subgraph with chromatic number at least $\sqrt{n}$. Indeed, the complement of the graph $G_g$ consisting of green edges is the graph $G_{rb}$ consisting of red and blue edges, and every graph $G$ satisfies $\chi(G)\cdot\chi(G^c)\geq |V(G)|$. It is not hard to construct an infinite set of examples on $n^2$ vertices for some integer $n$ where all three $2$-colored subgraphs have chromatic number precisely $n$, hence this is best possible. Our first result is that in the case of Gallai-colorings, we can do much better.

\begin{thm}\label{chrom1}
Every Gallai-$3$-coloring on $n$ vertices contains a $2$-colored subgraph that has chromatic number at least $n^{2/3}$.
\end{thm}

Fox--Grinshpun--Pach further obtained a generalisation of Theorem \ref{fgp1} to more colors.

\begin{thm}[Fox--Grinshpun--Pach, \cite{fox}]\label{fgp2}
Let $r$ and $s$ be fixed positive integers with $s\leq r$.  Every Gallai-$r$-coloring on $n$ vertices contains a set of order $\Omega\left(n^{\binom{s}{2}/\binom{r}{2}}\log^{c_{r,s}}n\right)$ which uses at most $s$ colors, and this bound is tight up to a constant factor. Here $c_{r,s}$ is a constant depending only on $r$ and $s$.
\end{thm}

Moreover, the value of $c_{r,s}$ was exactly determined in \cite{fox}. We prove a corresponding theorem about subgraphs with few colors and large chromatic number.

\begin{thm}\label{chrom2}
Let $r$ and $s$ be fixed positive integers with $s\leq r$.  Every Gallai-$r$-coloring on $n$ vertices contains an $s$-colored subgraph that has chromatic number at least $n^{s/r}$.
\end{thm}

Both Theorems \ref{chrom1} and \ref{chrom2} are sharp, as seen in Construction~\ref{eszconstrueasy}. The motivation of Fox--Grinshpun--Pach for proving Theorem \ref{fgp2} was to get one step closer towards proving the Erd\H{o}s--Hajnal conjecture. On the other hand, our main motivation for proving Theorem \ref{chrom2} came from a completely different direction - we tried to give an answer to Loh's question (see Section~\ref{lohquestionsection}) - in fact, when proving our main results we were not even aware of the Fox--Grinshpun--Pach paper. As it happens, we were only able to give a partial answer to Loh's question, and quite surprisingly, our proof of Theorem \ref{chrom2} was very similar to their proof of a weaker version of Theorem \ref{fgp2}: we use Gallai's structure theorem for Gallai-colored complete graphs to obtain a nice block partition of the vertex set, and then the theorem follows by induction with some more work.

\subsection{Long subchromatic paths in tournaments and Loh's question}\label{lohquestionsection}

A celebrated theorem of Erd\H{o}s and Szekeres \cite{esz} states that for any two positive integers $r,s$, every sequence of $rs+1$ (not necessarily distinct) numbers contains a monotone increasing subsequence of length $r+1$ or a monotone decreasing subsequence of length $s+1$. A short proof of this theorem is given by the pigeonhole principle. Assign to each number in the sequence an ordered pair $(x,y)$ where $x$ is the length of the longest increasing subsequence ending at this number, and $y$ is the decreasing analogue. It is then easy to see that all these ordered pairs have to be distinct, and therefore there is an ordered pair with first element at least $r+1$ or second element at least $s+1$.

The exact same proof also gives the following extension. Consider the $n$-vertex transitive tournament $T_n$, where the edge between $i<j$ is oriented in the direction $\overrightarrow{ij}$. Then every $2$-coloring of the edges of $T_n$ has a monochromatic directed path of length at least $\sqrt{n}$ (throughout this paper the length of a path means vertex-length, i.e. the number of vertices in the path). Moreover if we consider $r$-colored tournaments then the same proof shows that there exists a monochromatic tournament of length $n^{1/r}$. In fact, all the above results are sharp for infinitely many $n$, and they are also true for non-transitive tournaments.

Loh \cite{loh} asked about the following beautiful generalisation of the above. Determine $f(n,r,s)$, the maximum number such that every $r$-coloring of the edges of the transitive tournament on $n$ vertices contains a directed path with at least $f(n,r,s)$ vertices whose edges have at most $s$ colors. By grouping together sets of $s$ colors, a similar argument as the above shows that $$n^{1/\lceil r/s \rceil}\leq f(n,r,s),$$
and a standard construction shows that the upper bound
$$f(n,r,s)\leq n^{s/r}$$
holds whenever $n$ is a perfect $r$-th power.
However, already the $r=3,s=2$ case is non-trivial, since the above bounds only give $\sqrt{n}\leq f(n,3,2)\leq n^{2/3}(1+o(1))$. Loh proved that the correct answer in this case is not $\sqrt{n}$. Here, $\log^*$ is the iterated logarithm, or the inverse of the tower function $T(n) = 2^{T(n-1)}$, $T(0) = 1$.
\begin{thm}[Loh, \cite{loh} \label{lohsresultlogstar}] There exists a constant $C>0$ such that
$$C\sqrt{n}\cdot e^{\log^*n}\leq f(n,3,2),$$
that is, every $3$-coloring of the edges of the transitive $n$-vertex tournament contains a directed path of length at least $C\sqrt{n}\cdot e^{\log^*n}$ whose edges use at most $2$ colors.
\end{thm}

The question of determining the order of magnitude of $f(n,3,2)$, and in general that of $f(n,r,s)$, is still wide open. A direct corollary of Theorem \ref{chrom2} is a partial answer to Loh's question, which was the main motivation of this paper. Recall that a \emph{rainbow triangle} is a triangle whose edges are all different colors. A \emph{Gallai}-$r$-\emph{coloring} of $K_n$ is an $r$-coloring of the edges of $K_n$ without rainbow triangles.

\begin{thm}\label{poshenmainthm}
Let $r,s,n$ be positive integers with $s\leq r$.
Every Gallai-$r$-coloring of an $n$-vertex tournament contains a directed path on at least $n^{s/r}$ vertices, whose edges use at most $s$ colors.
\end{thm}

Note that every $2$-coloring is rainbow-triangle free, hence Theorem \ref{poshenmainthm} is a generalisation of the Erd\H{o}s--Szekeres Theorem.  We emphasize that our result holds for non-transitive tournaments as well, but our proof methods completely break down for colorings that contain rainbow triangles. 

Nevertheless, our guess is that  $f(n,r,s)=n^{s/r}$ holds whenever $n$ is a perfect $r$-th power. A reason for this is as follows. Suppose $r=3$ and $s=2$, and $u,v,w$ are vertices of the tournament such that $\overrightarrow{uv}$ is red, $\overrightarrow{vw}$ is blue and $\overrightarrow{uw}$ is green. Consider what happens if we recolor the $\overrightarrow{uw}$ edge to become, say, color red. Then the lengths of the red-green paths do not change, as the set of red-green edges did not change. The length of the longest red-blue paths did not change, as any path that contained the $\overrightarrow{uw}$ edge could have instead contained the $\overrightarrow{uv}$ and $\overrightarrow{vw}$ edges, giving a longer path. Finally, the length of the longest blue-green paths did not increase, as the set of blue-green edges descreased by the $\overrightarrow{uw}$ edge. Hence, in some sense, destroying this $uvw$ rainbow triangle gave us a better coloring, with shorter two-colored paths. Unfortunately, it is not the case that by repeating moves like the above one can always transform the coloring into a Gallai-coloring, hence this paragraph is nothing more than a (more or less) convincing heuristic argument.

\section{The proof of Theorem \ref{poshenmainthm}}

We will deduce Theorem \ref{poshenmainthm} from Theorem \ref{chrom2}. 
The connection between chromatic number and longest paths in orientations of graphs is given by the Gallai-Hasse-Roy-Vitaver Theorem. 

\begin{thm}[\label{gallairoy}Gallai-Hasse-Roy-Vitaver, \cite{gallairoy}]
Every orientation of the edges of a graph $G$ has a directed path on at least $\chi(G)$ vertices.
\end{thm}

\begin{proof}[Proof of Theorem \ref{poshenmainthm}]
We are given a Gallai-$r$-coloring of an $n$-vertex tournament. By Theorem \ref{chrom2} we can find an $s$-colored subgraph $G$ that has chromatic number at least $n^{s/r}$. By Theorem \ref{gallairoy} we can find a directed path $P$ in $G$ on at least $n^{s/r}$ vertices. As the edges of $G$ use at most $s$ colors, and $P$ is in $G$, we conclude that the edges of $P$ use at most $s$ colors and the proof is complete.
\end{proof}

A folklore construction shows that Theorems \ref{poshenmainthm} and \ref{chrom2} are sharp whenever $n$ is a perfect $r$-th power.

\begin{constru}\label{eszconstrueasy}
Consider the $n^r$ numbers $\{0,1,\ldots,n^r-1\}$ which will be the vertices of the graph. Write every such number as a $r$-digit base-$n$ number (by adding trailing zeros if necessary). Color edge $i<j$ according to the leftmost digit in which they differ and orient them as $\overrightarrow{ij}$ This $r$-colored transitive tournament has the property that the length of any $s$-colored path is at most $n^{s}$, and hence the chromatic number of the corresponding subgraph is at most $n^{s}$.
\end{constru}

\section{Proof of Theorem \ref{chrom2}}

We will deduce Theorem \ref{chrom2} from the following stronger statement. 

\begin{thm}\label{chromaticnumber}
Let $r,s,n\in\mathbf{Z}^+$ with $s\leq r$. Given a Gallai-$r$-coloring of $K_n$ using colors $[r]$ and given $S\subset [r]$ write $G_S$ for the subgraph whose edges are colored by elements of $S$. Then 
$$n^{\binom{r-1}{s-1}}\leq \prod_{|S|=s} \chi(G_S),$$
where the product goes over all $S\subset [r]$ with $|S|=s$. 
\end{thm}

The proof of Theorem \ref{chromaticnumber} shows a lot of similarities to the proof of Theorem $7.2.$ in \cite{fox}, the main new idea is in Claim \ref{beautyclaim}. Recall that Theorem \ref{chromaticnumber} is false for general colorings. Theorem \ref{chrom2} is an immediate consequence of Theorem \ref{chromaticnumber}.

\begin{proof}[Proof of Theorem \ref{chrom2}]
Given a Gallai-$r$-coloring on $n$ vertices, Theorem \ref{chromaticnumber} states that the geometric mean of the $\chi(G_S)$-s is at least $n^{s/r}$. Hence in particular there exists an $S\subset [r]$ with $|S|=s$ such that $\chi(G_S)\geq n^{s/r}$, as claimed.
\end{proof}

To prove Theorem \ref{chromaticnumber} we will need another theorem by Gallai:
\begin{lemma}[Gallai, \cite{gallaicolpaper}]\label{gallai}
An edge-coloring $F$ of a complete graph on a vertex set $V$ with $|V | \geq 2$ is a Gallai coloring if and only if $V$ may be partitioned into nonempty sets $V_1, \ldots , V_t$ with $t \geq 2$ so that each $V_i$ has no rainbow triangles under $F$, at most two colors are used on the edges not internal to any $V_i$, and the edges between any fixed pair $(V_i, V_j)$ use only one color. Furthermore, any such substitution of Gallai colorings for vertices of a $2$-edge-coloring of a complete graph $K_t$ yields a Gallai coloring.
\end{lemma}
For some recent progress on generalising Lemma \ref{gallai}, we direct the reader to the beautiful paper of Leader--Tan \cite{leadertan}. We will also make use of the following observation.
\begin{obs}\label{replacingchromthingie}
Let $G$ be a graph on vertex set $V(G)$ and let $V_1\cup V_2\cup \ldots \cup V_m$ be a partition of $V(G)$ such that for each pair of distinct $i,j\in [m]$, either all edges of the form $\{uv : u\in V_i, v\in V_j\}$ are present in $G$, or none of them. For each $i\in [m]$, let $G'_i$ be an arbitrary graph with chromatic number $\chi(G'_i)=\chi(G[V_i])$. Let $H$ be the graph obtained from $G$ by replacing $G[V_i]$ by $G'_i$ for each $i\in [m]$. Then $\chi(G)=\chi(H)$.
\end{obs}
The following is a common generalisation of H\"{o}lder's inequality that we will find useful.
\begin{lemma}\label{holderineq}
If $\mathcal{F}$ is a finite set of indices and for each $S\in\mathcal{F}$ we have that $a_S$ is a function mapping $[m]$ to the non-negative reals, then
$$\prod_{S\in\mathcal{F}}\sum_i a_S(i)\geq\left(\sum_i\prod_{S\in\mathcal{F}}a_S(i)^{1/|\mathcal{F}|}\right)^{|\mathcal{F}|}.$$
\end{lemma}

\begin{proof}[Proof of Theorem \ref{chromaticnumber}]
We prove the theorem by induction on $n$. If $n=1$ then every $\chi(G_S)$ is equal to $1$, and $n^{\binom{r-1}{s-1}}$ is also equal to $1$. If $n>1$, we can find a pair of colors $Q$ and some non-trivial partition of the vertices $V_1,\ldots,V_m$ such that for each pair of distinct $i,j$ in $[m]$, there is a $q\in Q$ so that all of the edges between $V_i$ and $V_j$ have color $q$. For each $S\subset [r]$ and $i\in[m]$ let $G_{S,i}$ be the subgraph of the complete graph on $V_i$ consisting of edges colored by colors in $S$. Write $\chi(S,i):=\chi(G_{S,i})$ and $\chi(S)=\chi(G_S)$.
Let $Q=\{q_1,q_2\}$. If $q_1\in S$ and $q_2\notin S$ then let $S^*:=S\cup\{q_2\}\setminus \{q_1\}$. 
\begin{claim}\label{beautyclaim} If $S\subset [k]$ with $q_1\in S$ and $q_2\notin S$, then
$$\chi(S)\chi(S^*)\geq \sum_{i=1}^m \chi(S,i)\chi(S^*,i).$$
\end{claim}
\begin{proof}[Proof of claim:]
For each $i\in[m]$, let $G_i'$ be the $2$-colored complete graph on $\chi(S,i)\chi(S^*,i)$ vertices, obtained by taking $\chi(S,i)$ disjoint copies of $K_{\chi(S^*,i)}$, coloring all edges inside the cliques by color $q_2$, and edges between the cliques by color $q_1$. For each $i$, replace $G[V_i]$ by $G'_i$, to obtain a $2$-colored complete graph $G'$ on $\sum_{i=1}^m \chi(S,i)\chi(S^*,i)$ vertices - let $H_1$ and $H_2$ be the subgraphs induced by edges of color $q_1$ and $q_2$ respectively. Note that $\chi(H_1)=\chi(S)$ and $\chi(H_2)=\chi(S^*)$ by Observation~ \ref{replacingchromthingie}, and since $H_2$ is the complement of $H_1$ we also have $\chi(S)\chi(S^*)=\chi(H_1)\chi(H_2)\geq |V(G')|=\sum_{i=1}^m \chi(S,i)\chi(S^*,i)$.
\end{proof}

In what follows we will always omit writing $|S|=s$ in the subscripts of products for clearer presentation. By induction, for all $i$ we have 
\begin{equation}\label{inductioneq}
|V_i|^{\binom{r-1}{s-1}}\leq \prod_{S} \chi(S,i).
\end{equation}
 Note that 
\begin{itemize}
\item if $Q\cap S = \emptyset $ then $\chi(S)=\max_i\{\chi(S,i)\}$,
\item If $Q\subseteq S$ then $\chi(S)=\sum_i(\chi(S,i))$ where the sum is over all $i\in [m]$.
\end{itemize}
Hence we have, using Claim \ref{beautyclaim}, that
\begin{equation}\label{claimapplied}
\prod_{S}\chi(S)\geq \left(\prod_{S:Q\cap S=\emptyset} \chi(S)\right)\left(\prod_{S:q_1\in S,q_2\notin S} \left(\sum_{i=1}^m\chi(S,i)\chi(S^*,i)\right)\right)\left(\prod_{S:Q\subseteq S}\sum_{i=1}^m \chi(S,i) \right).
\end{equation}

To simplify notation, if $q_1\in S$ and $q_2\notin S$ then write $\alpha(S,i):=\chi(S,i)\chi(S^*,i)$, and if $Q\subseteq S$ then write $\alpha(S,i):=\chi(S,i)$.  Let $\mathcal{F}=\{S\subset [r]: |S|=s, q_1\in S\}$. So 
\begin{equation*}
\begin{split}
&\prod_{S}\chi(S)\geq  \left(\prod_{S\in \mathcal{F}} \sum_{i=1}^m \alpha(S,i)\right)\prod_{S:Q\cap S=\emptyset} \chi(S)\geq \left(\sum_{i=1}^m \left( \prod_{S\in\mathcal{F}} \alpha(S,i)^{1/|\mathcal{F}|}\right)\right)^{|\mathcal{F}|}\prod_{S:Q\cap S=\emptyset} \chi(S) =\\
& \left(\sum_{i=1}^m \left(\prod_{S:Q\cap S=\emptyset} \chi(S) \prod_{S\in\mathcal{F}} \alpha(S,i)\right)^{1/|\mathcal{F}|}\right)^{|\mathcal{F}|}\geq 
 \left(\sum_{i=1}^m \left(\prod_{S}\chi(S,i)^{1/|\mathcal{F}|}\right)\right)^{|\mathcal{F}|}\geq \left(\sum_{i=1}^m |V_i|\right)^{|\mathcal{F}|}=n^{\binom{r-1}{s-1}},
\end{split}
\end{equation*}
where the first inequality is by rewriting (\ref{claimapplied}), the second inequality is by Lemma \ref{holderineq}, the third inequality is by $\chi(S)\geq \chi(S,i)$, and the fourth inequality is by (\ref{inductioneq}), using that $|\mathcal{F}|=\binom{r-1}{s-1}$. This completes the proof.
\end{proof}

\noindent\textbf{Note added after the refereeing process:} Recently Gowers--Long~\cite{gowerslong} improved Theorem~\ref{lohsresultlogstar} by showing that there exists $\eps>0$ such that $n^{0.5+\eps}\leq f(n,3,2)$.

\section{Acknowledgement}

The author is very grateful to J\'{o}zsef Balogh, Po-Shen Loh and Andrew Suk for many helpful discussions, and to Emily Heath and Ruth Luo for careful reading of the manuscript.

\end{document}